\date{November 4, 2014}

\documentclass[12pt]{amsart}
\usepackage{latexsym,amsmath,amsfonts,amscd,amssymb}
\usepackage{graphics}
\newtheorem{theorem}{Theorem}[section]

\newtheorem{corollary}[theorem]{Corollary}

\theoremstyle{remark}

\newcommand{\EE}{{\mathbb E}}

\newcommand{\RR}{{\mathbb R}}

\newcommand{\eps}{\epsilon}

\title{Kelly criterion for variable pay-off}


\author[R. P\'{e}rez Marco]{Ricardo P\'{e}rez Marco}
\address{CNRS, LAGA UMR 7539, Universit\'e Paris XIII,
99, Avenue J.-B. Cl\'ement, 93430-Villetaneuse, France}
\email{ricardo@math.univ-paris13.fr}

\thanks{.}

\begin{document}

\maketitle

\begin{abstract}
We determine Kelly criterion for a game with variable pay-off. The Kelly fraction satisfies a fundamental integral equation and is 
smaller than the classical Kelly fraction for the same game with the constant average pay-off.
\end{abstract}

\section{Introduction.}

Kelly criterion (see \cite{Ke}), also called "`Fortune Formula"', is the fundamental tool 
to ensure an optimal positive return when playing with repetition a favorable game or strategy. 
It was first proposed by Edward Thorp for the money management of his card counting strategy 
to beat casino blackjack (see \cite{Th1}). Later he applied the same ideas to the financial markets (see \cite{Ka-Th}), and L. Breinman (see \cite{Bre})
proved that it was the optimal strategy for long run accumulation of capital.  

One needs to modify the theory for common situation where the advantage is not known exactly 
but is a random variable (see \cite{MG-PM}). In this
"`fuzzy advantage"' situation one needs to be far more conservative (see section 3 in \cite{MG-PM}).
When we apply Kelly criterion to decision making in real situations, we are confronted with a fuzzy advantage 
but also a variable reward or pay-off. In this article we study how to modify Kelly criterion when the pay-off
is a random variable with a known distribution. We obtain the same type of conclusion than in \cite{MG-PM}: The optimal Kelly fraction 
is more conservative than the one in the same game with the constant average pay-off. 

\section{The classical Kelly criterion.}

We assume that we are playing a game with repetition. At each round we risk a fraction $0\leq f\leq 1$ of our capital $X$. 
With probability $0 < p <1$ we win and the pay-off is $b$-to-$1$, with $b\geq 0$. This means that if $X$ is 
our current capital, if we lose the bet (which happens with probability
$q=1-p$) we substract $fX$ to our capital, and if we win, we add $bfX$ to our capital . Thus, the expected gain is
$$
\EE (\Delta X) = p b f X - q fX= (p b-q) fX = (p(1+b)-1) fX \ .
$$
If we play a game with advantage, that is $p(1+b)>1$, then we expect an exponential growth 
of our initial bankroll $X_0$ if we follow a reasonable
betting strategy. We assume that there is no minimal 
unit of bet. By homogeneity of the problem, a sharp 
strategy must consist in betting a proportion $f(p)$ of 
the total bankroll. 
In the classical Kelly criterion, $p$ and $b$ are assumed to be constant at each round. 
In this article we assume that $p$ is constant but the pay-off $b$ is a random variable. We first review the classical Kelly criterion 
with a constant pay-off that finds $f(p)$ in order 
to maximize the expected exponential growth. Our bankroll after having played $n$ rounds of the 
game have is
$$
X_n=X_0 \prod_{i=1}^n (1+b_i f(p))
$$
where $b_i=b$ if we won the $i$-th round, and $\eps_i=-1$
if we lost the $i$-th round. The exponential rate of growth of the bankroll is
$$
G_n =\frac1n \log \frac{X_n}{X_0}=
\frac1n \sum_{i=1}^n \log (1+b_i f(p)).
$$
The Kelly criterion maximizes the expected value of the 
exponential rate of growth:

\begin{theorem} \textbf {(Kelly criterion)}
For a game with advantage, that is $p(b+1) >1$, the expected value of the exponential rate 
of growth $G_n$ is maximized for 
$$
f^*(p, b)=p-q/b= \frac{p(1+b)-1}{b}\ .
$$
\end{theorem}

The argument is straightforward. Observe that the expected value is
$$
\EE (G_n)= \EE (G_1)=p\log (1+bf)+(1-p) \log (1-f)=g(f).
$$
This function of the variable $f$ has a derivative,
$$
g'(f)=\frac{pb}{1+bf} - \frac{1-p}{1-f}=\frac{(p(1+b)-1)-fb}{(1+fb)(1-f)} \ ,
$$
and $g(f)\to 0$ when $f\to 0^+$, and $g(f) \to - \infty$ when $f\to 1^-$. Also
$g'(f) >0$ near $0$, and $g'$ is decreasing, so $g$ is concave.
Thus $g(f)$ increases from $0$ to its maximum attained at 
$$
f^*=f^*(p, b)=\frac{p(1+b)-1}{b}
$$
and then decreases to $-\infty$.

\section{Kelly criterion with variable pay-off.}

This situation arises in a number of practical situations. The original problem that motivated the use of 
Kelly criterion was casino blackjack where the pay-off is constant $b=1$. But in other card games, like poker cash, 
the pay-off is variable.  Also some trading strategies cannot set a predetermined pay-off, for 
example when speculating with a price rebound in a volatile market. The historic of trades of traders present a certain distribution 
of pay-offs for the successful trades. Therefore in these situations we cannot consider $b$ constant. We assume that the 
pay-off $b$ is a random variable with a known non-negative distribution $\rho: \RR_+ \to \RR_+$, with $\rho(x) dx$ giving the probability that 
the pay-off is in the infinitesimal interval $[x, x+dx]$. In practice $\rho$ has compact support and can be obtained empirically, although the model in 
\cite{PM} shows that the tail is of Pareto type. 
We do not need to assume that the distribution is absolutely continuous with respect to the Lebesgue measure (same proofs). 

We can use a similar argument as in the previous section in order to determine the sharp fraction to bet $\hat f$. 

To determine when the game is favorable
we compute
$$
\EE(X_1-X_0)= (1-p) f X_0  + p f X_0 \int_0^{+\infty } b \ \rho (b) \ db \ , 
$$
Thus the condition for a favorable game is
\begin{equation} \label{advantage}
p\left ( 1 + \int_0^{+\infty } b \ \rho (b) \ db \right ) >1 \ ,
\end{equation}
which (naturally) is the same condition than that of a constant pay-off game where the pay-off is the average pay-off:
$$
p( 1 + \bar b ) >1 \ ,
$$
and 
$$
\bar b =\int_0^{+\infty} b \ \rho (b) \ db.
$$

After $n$ rounds, if $b_i$ is the 
pay-off in the $i$-th round, the expected 
exponential growth is
\begin{align*}
g(f)=\EE (G_n) &=\frac1n \sum_{i=1}^n \EE (\log (1+b_i f)) \\
&= q \log (1-f) + p \int_0^{+\infty} \log(1+b f) \ \rho(b) \ db  \ ,
\end{align*}
under the integrable assumption that the density of the pay-off function makes the integral finite.

Again $g(f) \to 0$ when $f\to 0^+$, $g(f) \to -\infty$ when $f\to 1^-$, and
$$
g'(f)= - \frac{q}{1-f} + p \ \int_0^{+\infty} \frac{b \rho (b)}{1+bf} \ db \ .
$$
We observe that when $f\to 0^+$,
$$
g'(f) \approx -(1-p) + p \int_0^{+\infty } b \ \rho (b) \ db 
$$
thus $g'(f) >0$ by the favorable game condition (\ref{advantage}). Moreover $g'$ is strictly decreasing, 
so $g$ is strictly concave, and tends to $-\infty$ when $f\to 1^-$. Thus there
is exactly one value $\hat f=\hat f(p)$ that maximizes this expression and annihilates $g'$. It is the unique solution $\hat f (p, \rho)$ to 
the fundamental integral equation
\begin{equation}\label{fundamental_equation}
p\int_0^{+\infty } \frac{b \ \rho (b)}{1+b \hat f } \ db-\frac{1-p}{1-\hat f} =0 \ .
\end{equation}

\begin{theorem}\textbf{(Kelly criterion for variable pay-off)}
A game with variable pay-off with a distribution $\rho$ is favorable if
\begin{equation*} 
p\left ( 1 + \int_0^{+\infty } b \ \rho (b) \ db \right ) >1 \ .
\end{equation*}
The expected value of the exponential rate of growth is maximized for $0<\hat f=\hat f(p, \rho) <1$ satisfying the fundamental integral 
equation
\begin{equation*}
p\int_0^{+\infty } \frac{b \ \rho (b)}{1+b \hat f } \ db-\frac{1-p}{1-\hat f} =0 \ .
\end{equation*}
\end{theorem}

\begin{corollary}
The optimal Kelly fraction $\hat f (p, \rho)$ for a favorable game with variable pay-off is smaller than the optimal Kelly fraction 
$f^*(p, \bar b)$ for the same game with constant pay-off equal to the average pay-off $\bar b$,
$$
\bar b =\int_0^{+\infty} b \ \rho (b) \ db \ .
$$
We have
$$
\hat f (p, \rho)\leq f^*(p,\bar b) \ .
$$
We only have equality when the pay-off is constant.
\end{corollary}

\begin{proof}
The function of $b$, 
$$
h(b)= \frac{b}{1+b\hat f} \ ,
$$
is strictly concave and therefore, by Jensen's inequality, we have 
$$
\int_0^{+\infty} \frac{b \ \rho(b)}{1+b\hat f} \ db \leq  \frac{\bar b}{1+\bar b f} \ .
$$
So, using the fundamental equation (\ref{fundamental_equation})we get
$$
p \frac{\bar b}{1+\bar b f}-\frac{1-p}{1-\hat f}= -\frac{1-p}{1-f^*}-\frac{1-p}{1-\hat f}\geq 0 \ ,
$$
and the result follows. The case of equality follows from the case of equality of Jensen's inequality and only occurs for a Dirac distribution. 
\end{proof}

\medskip

\textbf{Conclusion:}

In a situation when the pay-off is variable one needs to adjust the Kelly fraction in a conservative way. 

\textbf{Remark.}

The analysis generalizes to situations where the risk is larger than the fraction waged. As noted by Thorp, this happens in a leveraged 
investment in the financial markets.

\end{document}